\documentclass[12pt]{amsart}
\usepackage[inner=1in,outer=1in,top=1in,bottom=1in]{geometry}
\usepackage[utf8]{inputenc}
\usepackage[T1]{fontenc}
\usepackage{url,mathrsfs}
\usepackage{amsmath}
\usepackage{amssymb}
\usepackage{amsthm}
\usepackage[
  natbib=true,
  backend=bibtex,
  style=alphabetic,
  maxalphanames=99,
  maxbibnames=99,
  sorting=nyt,
]{biblatex}
\usepackage{graphicx, epstopdf}
\usepackage{subfig}
\usepackage{color}
\usepackage{bbm} 
\usepackage{dsfont}
\usepackage[shortlabels]{enumitem}
\usepackage{subfloat}
\usepackage[draft]{fixme}
\usepackage{bm}
\usepackage{bbm} 
\usepackage{ifpdf}
\usepackage{array}
\usepackage{mathtools}
\usepackage{multicol}
\usepackage{multirow}
\usepackage{graphicx}
\usepackage{tikz}
\usepackage{tikz-3dplot}
\usepackage{thmtools}
\usepackage{todonotes}
\usepackage{hyperref}
\usepackage[capitalize, nameinlink]{cleveref}
\hypersetup{
    colorlinks=true,
    citecolor=green,
    filecolor=black,
    linkcolor=blue,
    urlcolor=blue, % Change these colors as desired
    pdfauthor={Tracy Chin},
}

\newcommand\CC{{\mathbb C}}

\newcommand\RR{{\mathbb R}}

\newcommand\cal{\mathcal}

\newcommand\cB{{\cal B}}

\renewcommand{\phi}{\varphi}
\renewcommand{\epsilon}{\varepsilon}

\newcommand{\dl}{\partial}

\newcommand\conv{\operatorname {conv}}

\DeclareMathOperator{\sign}{sign}
\newcommand\supp{\operatorname {supp}}

\DeclareMathOperator{\val}{val}

\renewcommand{\Im}{\operatorname{Im}}

\newcommand\lbbrace{\{\!\{}
\newcommand\rbbrace{\}\!\}}

\theoremstyle{definition}
\newtheorem{definition}{Definition}[section]

\newtheorem{example}[definition]{Example}
\newtheorem*{example*}{Example}

\theoremstyle{plain}
\newtheorem{theorem}[definition]{Theorem}
\newtheorem*{theorem*}{Theorem}
\newtheorem{lemma}[definition]{Lemma}
\newtheorem{proposition}[definition]{Proposition}

\newtheorem*{corollary*}{Corollary}

\title{Oriented and Valuated Delta Matroids from Stable Polynomials}
\author{Tracy Chin}
\date{September 8, 2026}

\begin{document}
  \begin{abstract}
    Stable polynomials are the natural multivariate generalization of real rooted univariate polynomials. While their definition is purely algebraic in nature, they have deep connections to combinatorics. One such connection is the support theorem proved by \citeauthor{branden2007halfplane} in \citeyear{branden2007halfplane}, showing that the support of any stable polynomial is a jump system, and hence that the support of any multiaffine stable polynomial is a $\Delta$-matroid. In this work, we generalize this result, showing that coefficients of multiaffine real stable polynomials give rise to oriented $\Delta$-matroids, and coefficients of multiaffine stable polynomials over Puiseux series give rise to valuated $\Delta$-matroids.
  \end{abstract}
  \maketitle
  
  \section{Introduction}
  Stable polynomials are the natural multivariate generalization of real-rooted univariate polynomials. They have been central to many mathematical discoveries in seemingly unrelated areas, including constructing infinite families of Ramanujan graphs \cite{marcus2015interlacingi} and resolving the Kadison-Singer problem \cite{marcus2015interlacingii}.
  In a seminal work on stability, \citeauthor{branden2007halfplane} established a combinatorial support condition for stable polynomials. In this work, we expand on the multiaffine corollary of his main result.
  \begin{theorem*}[{\cite[Corollary 3.3]{branden2007halfplane}}]
    The support of a multiaffine stable polynomial is a $\Delta$-matroid.
  \end{theorem*}

  More recently, generalizations of $\Delta$-matroids have become a greater focus of study, paralleling similar classes for ordinary matroids. In \cite{booth2001orientedlagrangian}, \citeauthor{booth2001orientedlagrangian} establish a theory of oriented $\Delta$-matroids. (Note that the Lagrangian matroids in \cite{booth2001orientedlagrangian} are synonymous with $\Delta$-matroids. Throughout this paper we use the $\Delta$-matroid terminology.) Recent papers have also proferred two definitions for valuated $\Delta$-matroids: a definition based on discrete geometry and representability in \cite{cheung2025valuateddeltamatroidsprincipal}, and one arising from the broader theory of Coxeter matroids in \cite{calvert2025quadraticexchangeequationscoxeter}.

  In this work, we establish signed and valuated versions of \citeauthor{branden2007halfplane}'s support theorem. Specifically, for the former, we show that sign patterns of real stable polynomials give rise to oriented $\Delta$-matroids.
  \begin{theorem*}[\cref{thm:real-stable-implies-orientation}]
    If $f = \sum_{S \subseteq [n]} a_S x^S$ is a real stable polynomial with $a_\varnothing > 0$, then $s_\varnothing : S \mapsto \sign(a_S)$ is an orientation of $\supp(f)$ with respect to $\varnothing$.
  \end{theorem*}
  Since determinantal polynomials are always real stable, this result is a generalization of the results from \cite[Section 4]{booth2001orientedlagrangian} and \cite{boege2026signpatterns} from the determinantal setting to arbitrary real stable polynomials.

  We also establish a connection between stable polynomials over Puiseux series and valuated $\Delta$-matroids.
  \begin{theorem*}[\cref{thm:stable-implies-val-delta-matroid}]
    If $f = \sum_{S\subseteq [n]} c_S x^S \in \CC\lbbrace t\rbbrace[x_1,\dots,x_n]$ is stable, then $\nu: S \mapsto \val(c_S)$ is a valuated $\Delta$-matroid in the sense of \cite{cheung2025valuateddeltamatroidsprincipal}.
  \end{theorem*}
  This again connects to a result from \cite{cheung2025valuateddeltamatroidsprincipal} showing that valuations of determinantal polynomials give valuated $\Delta$-matroids, and is a natural extension of \cite[Theorem 4]{branden2010discreteconcavity} beyond the constant-parity case. However, as shown in \cref{ex:no-strong-exchange}, supports of stable polynomials need not satisfy the strong exchange property, so these valuations do not satisfy the definition of valuated $\Delta$-matroids described in \cite{calvert2025quadraticexchangeequationscoxeter}.

  \section{Preliminaries}
  \subsection{Delta Matroids}
  We begin by recalling the definition and some basic facts about $\Delta$-matroids.
  \begin{definition}
    A \emph{$\Delta$-matroid} is a non-empty collection $\cB \subseteq 2^{[n]}$ such that for any $X,Y \in \cB$ and any $i \in X \Delta Y$, there exists $j \in X \Delta Y$ such that $X \Delta\{i,j\}\in\cB$.

    We say a $\Delta$-matroid has the \emph{strong exchange property} if for all $X,Y \in \cB$ and $i \in X \Delta Y$ there exists $j \in X\Delta Y$ such that $X\Delta \{i,j\}$ and $Y\Delta \{i,j\}$ are both in $\cB$.
  \end{definition}
  Note that unlike matroids, not all $\Delta$-matroids have the strong exchange property.

  \begin{example}
    Consider the $\Delta$-matroid with bases $\cB = \{\varnothing, 1, 2, 3, 123\}$. One can check that this is a $\Delta$-matroid. However, it does not have the strong exchange property. If $X = \varnothing, Y = 123$, there are no $i,j$ such that $X \Delta \{i,j\}$ and $Y\Delta \{i,j\}$ are both in $\cB$.
  \end{example}

  Similarly to matroids, $\Delta$-matroids can also be characterized in terms of polytopes.
  \begin{proposition}
    A nonempty collection $\cB\subseteq [n]$ is a $\Delta$-matroid if and only if all edges of $\conv\{e_S : S \in \cB\}$ are in directions $\pm e_i, e_i \pm e_j$.
  \end{proposition}
  In particular, this means that every edge in a $\Delta$-matroid polytope is contained within some two-dimensional face of the hypercube.

  In this paper, we consider two generalizations of $\Delta$-matroids, namely oriented $\Delta$-matroids and valuated $\Delta$-matroids. Oriented $\Delta$-matroids were first defined in \cite{booth2001orientedlagrangian}, while the definition valuated $\Delta$-matroids that we will use comes from \cite{cheung2025valuateddeltamatroidsprincipal}.

  \subsubsection{Oriented $\Delta$-matroids}
  Oriented $\Delta$-matroids were first defined in \cite{booth2001orientedlagrangian} under the name \emph{oriented Lagrangian matroids}. In order to present their definition, we must discuss the edges and faces that can arise in a $\Delta$-matroid polytope.

  First, we fix a reference basis $F \in \cB$. If $\cB$ is a $\Delta$-matroid, then all edges in the polytope $P_\cB = \conv\{e_B : B \in \cB\}$ have length one or two (with respect to the $\ell_1$-norm). Call edges of length one \emph{short edges} and edges of length two \emph{long edges}.

  \begin{definition}[\cite{booth2001orientedlagrangian}]
    Suppose $[e_A, e_B]$ is a long edge in $P_\cB$. We say that it is a \emph{horizontal long edge} (with respect to $F$) if $|A \Delta F| = |B \Delta F|$. Otherwise, we say that it is a \emph{vertical long edge} (with respect to $F$).
  \end{definition}

  This now allows us to define orientations with respect to $F$.
  \begin{definition}[\cite{booth2001orientedlagrangian}]\label{def:delta-matroid-orientation}
    We say a function $s_F: \cB \to \{\pm 1\}$ is an \emph{orientation with respect to $F$} if the following axioms are satisfied:

    \paragraph{\textsc{Axiom 1.}} $\cB$ is a $\Delta$-matroid.

    \paragraph{\textsc{Axiom 2.}} If there is a horizontal long edge (relative to $F$) between $A$ and $B$, then $s_F(A) = s_F(B)$.

    \paragraph{\textsc{Axiom 3.}} If there is a vertical long edge (relative to $F$) between $A$ and $B$, then $s_F(A) = -s_F(B)$.

    \paragraph{\textsc{Axiom 4.}} In a two-dimensional face of the hypercube, if three of the vertices have the same sign, then the vertex with the opposite sign is either the closest or furthest vertex from $F$.

    \paragraph{\textsc{Axiom 5.}} $s_F(F) = +1$.
  \end{definition}

  Importantly, note that Axioms 2-4 in \cref{def:delta-matroid-orientation} can each be checked by considering only two-dimensional faces of the hypercube $[0,1]^n$.

  In order to make this definition independent of our reference basis, we extend to a function on $\cB \times \cB$.
  \begin{definition}[\cite{booth2001orientedlagrangian}]\label{def:oriented-delta-matroid}
    We say the pair $(\cB,s)$ with $\cB\subseteq 2^{[n]}$ and $s: 2^{[n]} \times 2^{[n]} \to \{+1, -1, 0\}$ is an \emph{oriented $\Delta$-matroid} if it satisfies all of the following properties:
    \begin{itemize}
      \item $\cB$ is a $\Delta$-matroid.
      \item $s = 0$ whenever either of its arguments is not a basis of $\cB$.
      \item For all $F \in \cB$, $s_F(\cdot) \coloneqq s(F, \cdot)$ is an orientation of $M$ with respect to $F$.
      \item For all $F,G,H \in \cB$, we have
        \[s(G,H) = (-1)^{|G \setminus (F \cup H)|} s(F, G) s(F, H).\]
    \end{itemize}
  \end{definition}

  Every orientation of $M$ with respect to a reference basis can be extended to an oriented $\Delta$-matroid by taking
  \[s(G,H) = (-1)^{|G \setminus (F \cup H)|}s_F(G) s_F(H).\]
  We say two orientations $s_F, s_G$ with respect to reference bases $F, G$ are \emph{equivalent} if they extend to the same $s$.

  \subsubsection{Valuated $\Delta$-matroids}
  Two different definitions of valuated $\Delta$-matroids were introduced in \cite{cheung2025valuateddeltamatroidsprincipal} and \cite{calvert2025quadraticexchangeequationscoxeter}, respectively. In this paper, we use the \cite{cheung2025valuateddeltamatroidsprincipal} definition, as our $\Delta$-matroids do not necessarily satisfy the strong exchange axiom.

  \begin{definition}[\cite{cheung2025valuateddeltamatroidsprincipal}]
    A function $p: \{0,1\}^n \to \RR\cup\{\infty\}$ is a \emph{valuated $\Delta$-matroid} if every cell in the induced regular subdivision is a $\Delta$-matroid polytope. Equivalently, $p$ is a valuated $\Delta$-matroid if and only if all the edges of its induced subdivision have length at most two.
  \end{definition}

  \subsection{Stable Polynomials}
  The other object under study in this paper is stable polynomials.

  \begin{definition}
    We say a polynomial $f \in \CC[x_1,\dots,x_n]$ is \emph{stable} if $f(z_1,\dots,z_n) \neq 0$ whenever $\Im(z_1),\dots,\Im(z_n) > 0$. If furthermore $f$ has real coefficients, we say it is \emph{real stable}.
  \end{definition}

  A classical result of \citeauthor{branden2007halfplane} shows that the support of any multiaffine stable polynomial is a $\Delta$-matroid \cite[Corollary 3.3]{branden2007halfplane}. In the same paper, \citeauthor{branden2007halfplane} also gives a characterization for testing whether a multiaffine polynomial is real stable using Rayleigh differences. (Recall that we say a polynomial is \emph{multiaffine} if it has degree at most one in each variable.)

  \begin{definition}
    For a polynomial $f \in \CC[x_1,\dots,x_n]$ and $1 \leq i,j\leq n$, we define the \emph{Rayleigh difference}
    \[\Delta_{ij}f = \frac{\dl f}{\dl x_i} \frac{\dl f}{\dl x_j} - f \frac{\dl^2 f}{\dl x_i \dl x_j}.\]
  \end{definition}

  Real stability is then characterized by global nonnegativity of all Rayleigh differences.
  \begin{theorem}[{\cite[Theorem 5.6]{branden2007halfplane}}]
    Let $f \in \RR[x_1,\dots, x_n]$ be multi-affine. Then $f$ is real stable if and only if for all $x \in \RR^n$ and all $1 \leq i,j\leq n$,
    \[\Delta_{ij}f(x) \geq 0.\]
  \end{theorem}

  In particular, we will require the bivariate consequence of this statement.
  \begin{proposition}[{\cite[Example 5.7]{branden2007halfplane}}]
    Let $f = a + bx_1 + c x_2 + dx_1x_2 \in \RR[x_1,x_2]$. Then $f$ is real stable if and only if $bc-ad \geq 0$.
  \end{proposition}

  \section{Real Stability and Oriented Delta Matroids}
  In this section, we show that sign patterns of real stable polynomials give oriented $\Delta$-matroids. First, we prove that we get an orientation with respect to $F = \varnothing$ when fix our reference basis to be the empty set. Then, we show how real-stability preserving operations allow us to interpret extending to an oriented $\Delta$-matroid in the language of polynomials.

  \begin{theorem}\label{thm:real-stable-implies-orientation}
    Let $f = \sum_{S\subseteq [n]}a_S x^S$ be a multiaffine polynomial, and consider the map $\sigma: \supp(f) \to \{\pm 1\}$ given by $S \mapsto \sign(a_S)$. If $f$ is real stable and $a_{\varnothing} > 0$, then $\sigma$ is an orientation of $\supp(f)$ with respect to $\varnothing$.
  \end{theorem}
  \begin{proof}
    By \cite[Corollary 3.3]{branden2007halfplane}, the support of $f$ is a $\Delta$-matroid, satisfying the first condition in \cref{def:delta-matroid-orientation}. Furthermore, we assume that $a_\varnothing > 0$, which automatically gives us Axiom 5, so it just remains to check Axioms 2-4.
    
    Since real stability is preserved under taking initial forms, and since it suffices to check each axiom on the two-dimensional faces of the hypercube, we may restrict ourselves to real stable polynomials of the form
    \[f = x^A(a + b x_i + c x_j + d x_i x_j)\]
    where $A \subseteq [n]$ and $i,j \notin A$.

    In this setting, and taking our reference basis to be $F = \varnothing$, we may rephrase the axioms as follows:

    \paragraph{\textsc{Axiom 2.}} If $a = d = 0$ and $b,c \neq 0$, then $\sign(b) = \sign(c)$.
    \paragraph{\textsc{Axiom 3.}} If $b = c = 0$ and $a,d \neq 0$, then $\sign(a) = -\sign(d)$.
    \paragraph{\textsc{Axiom 4.}} If $a,b,c,d$ are all nonzero and three of them have the same sign, then the coefficient with the opposite sign must be either $a$ or $d$.

    Since $f = x^A(a + b x_i + c x_j + d x_i x_j)$ is real stable if and only if $g = a + b x_i + c x_j + d x_ix_j$ is real stable, we may further restrict ourselves to the case where $A = \varnothing$. By \cite[Example 5.7]{branden2007halfplane}, we know that $g$ is real stable if and only if $bc - ad \geq 0$, from which all of the above properties follow.

    Hence, $\sigma: \supp(f) \to \{\pm 1\}$ is an orientation of $\supp(f)$ with respect to $F = \varnothing$, as desired.
  \end{proof}

  We can extend this to other reference bases using stability-preserving operations. First, we establish that the operation in question preserves stability.

  \begin{lemma}\label{lem:f-delta-G-coeffs}
    Let $f = \sum_{S \subseteq [n]} a_S x^S$ and fix $G \subseteq [n]$. If $f$ is real stable, then $f_{\Delta G} \coloneq \sum_{S\subseteq [n]} (-1)^{|G \setminus S|}a_{S} x^{S \Delta G}$ is also real stable. 
  \end{lemma}
  \begin{proof}
    We claim that we obtain $f_{\Delta G}$ by iteratively applying $f \mapsto -x_i f(x_1,\dots,\frac{-1}{x_i},\dots, x_n)$ for each $i \in G$. We note that $\Im(-1/x_i) > 0$ if and only if $\Im(x_i) > 0$; in particular, this operation preserves stability, so proving this claim suffices to complete the proof.

    Iteratively applying these operations for all $i \in G$ gives us $(-1)^{|G|}x^G f(z_1,\dots, z_n)$, where $z_i = -1/x_i$ if $i \in G$ and $z_i = x_i$ otherwise. Under this map, we have
    \begin{align*}
      x^A &\mapsto (-1)^{|G|}x^G\left(\prod_{i \in G\cap A}\frac{-1}{x_i}\right)\left( \prod_{i \in A \setminus G}x_i \right)\\
        &= (-1)^{|G|}(-1)^{|G\cap A|}x^{G\setminus(G \cap A)}x^{A \setminus G}\\
        &= (-1)^{|G \setminus A|} x^{A \Delta G}.
    \end{align*}
    Since the map is linear, this gives us $f \mapsto f_{\Delta G}$, as desired.
  \end{proof}

  In particular, this means that we can deduce $s: 2^{[n]} \times 2^{[n]} \to \{0, \pm 1\}$ via stability-preserving operations.
  \begin{theorem}
    Suppose $f = \sum a_S x^S \in \RR[x_1,\dots,x_n]$ is a multiaffine real stable polynomial. Then $s(G, H) = \sign(a_G)\sign(f_{\Delta G}[x^{H \Delta G}])$, where $f_{\Delta G}[x^{H \Delta G}]$ denotes the coefficient of $x^{H\Delta G}$ in $f_{\Delta G}$. In particular, we can obtain $s_G$ using only stability-preserving operations.
  \end{theorem}
  \begin{proof}
    From \cite[Theorem 3]{booth2001orientedlagrangian}, we have 
    \[s(G,H) = s(\varnothing, G)s(\varnothing, H)(-1)^{|G \setminus H|} = \sign(a_G)\sign(a_H)(-1)^{|G \setminus H|},\]
    which by \cref{lem:f-delta-G-coeffs} is equal to $\sign(a_G)\sign(f_{\Delta G}[x^{H \Delta G}])$, as desired.
  \end{proof}

  Thus, the orientation with respect to $G$ that we acquire from the operation in \cref{lem:f-delta-G-coeffs} is equivalent to the orientation defined by the coefficients of $f$.

  \section{Stability and Valuated Delta Matroids}
  Next, we turn our attention to valuated $\Delta$-matroids. In \cite[Theorem 4]{branden2010discreteconcavity}, \citeauthor{branden2010discreteconcavity} shows that a stable polynomial over $\CC\lbbrace t\rbbrace$ with constant parity support gives rise to an $M$-convex function. In particular, it follows that a multiaffine such polynomial induces a valuated even $\Delta$-matroid in the sense of \cite{rincon2012isotropicallinearspaces}. In that same paper, \citeauthor{branden2010discreteconcavity} also shows that real-stable polynomials with nonnegative coefficients induce $M^\natural$-convex functions.

  In this section, we generalize these results to arbitrary multiaffine stable polynomials, not just those with constant parity or nonnegative real coefficients. However, the discrete functions we obtain are less nicely behaved than those in \cite{branden2010discreteconcavity} in that they may not satisfy the strong exchange property (\cref{ex:no-strong-exchange}).

  \begin{theorem}\label{thm:stable-implies-val-delta-matroid}
    Let $f = \sum_{S \subseteq [n]} c_S x^S \in \CC\lbbrace t\rbbrace[x_1,\dots,x_n]$ be a multiaffine stable polynomial, and define $\nu: 2^{[n]} \to \RR \cup \{\infty\}$ by $\nu(S) = \val(c_S)$. Then $\nu$ is a valuated $\Delta$-matroid.
  \end{theorem}
  \begin{proof}
    As in \cite[Proposition 8]{branden2010discreteconcavity}, we note that we can consider stability as a first order statement over $\RR$ by identifying $\CC \cong \RR \times \RR$. Hence, Tarski's principle applies, and first-order statements about stable polynomials over $\CC$ translate to $\CC\lbbrace t\rbbrace$.

    We wish to show that every edge in the subdivision induced by $\nu$ has length at most two. Indeed, suppose that $[e_A, e_B]$ is an edge in the regular subdivision induced by $\nu$. By definition, there exists $w \in \RR^n$ such that $A,B$ are the unique subsets of $[n]$ attaining $\min_{S\subseteq [n]}\langle (w,1), (e_S, \nu(S))\rangle$. Define $m$ to be this minimum, so $\langle w, e_S\rangle + \val(c_S) \geq m$ for all $S \subseteq [n]$, with equality if and only if $S = A,B$.

    Since stability is preserved under global scaling and nonnegative scaling of variables, the polynomial
    \[g \coloneq t^{-m} f(t^{w_1}x_1,\dots, t^{w_n}x_n) \in \CC\lbbrace t\rbbrace[x_1,\dots,x_n]\]
    is stable. By Tarski's theorem, it follows that $g(t)$ is stable for all sufficiently small $t > 0$ (see, e.g., \cite[Section 1.5]{speyer2005horn} and references therein). Since stability is preserved under taking limits, by taking $t \to 0$, we obtain a stable binomial of the form $a x^A + bx^B$ for some $a,b \neq 0$. It then follows from \citeauthor{branden2007halfplane}'s support theorem \cite[Theorem 3.3]{branden2007halfplane} that $|A \Delta B| \leq 2$, as desired.
  \end{proof}

  However, we note that while stable polynomials with constant parity support or nonnegative real coefficients satisfy the strong exchange property, this is no longer true once we remove these assumptions. In particular, we cannot strengthen the conclusion of \cref{thm:stable-implies-val-delta-matroid} to $M$-convexity or $M^\natural$-convexity.
  \begin{example}\label{ex:no-strong-exchange}
    Consider
    \[f = \det\left(\begin{pmatrix}
      0 & 1 & 1 \\ 1 & 0 & 1 \\ 1 & 1 & 0
    \end{pmatrix} + \begin{pmatrix}
      x_1 & & \\ & x_2 & \\ & & x_3
    \end{pmatrix}\right) = 2 - x_1 - x_2 - x_3 + x_1x_2x_3.\]
    This is a determinantal polynomial and hence real stable. However, considered as a polynomial over $\RR\lbbrace t\rbbrace$, its tropicalization does not have the strong exchange property. In particular, valuations of stable polynomials do not necessarily satisfy the stronger notion of valuated $\Delta$-matroids defined in \cite{calvert2025quadraticexchangeequationscoxeter}.
  \end{example}
  \printbibliography

@article{boege2026signpatterns,
  author   = {Boege, Tobias and Selover, Jesse and Zubkov, Maksym},
  title    = {Sign patterns of principal minors of real symmetric matrices},
  journal  = {Linear Algebra Appl.},
  fjournal = {Linear Algebra and its Applications},
  volume   = {738},
  year     = {2026},
  pages    = {161--188},
  issn     = {0024-3795,1873-1856},
  mrclass  = {05B20 (14P10 14P25 15A15)},
  mrnumber = {5040123},
  doi      = {10.1016/j.laa.2026.02.029},
  url      = {https://doi.org/10.1016/j.laa.2026.02.029},
}

@incollection{booth2001orientedlagrangian,
  author     = {Booth, Richard F. and Borovik, Alexandre V. and Gelfand,
                Israel M. and White, Neil},
  title      = {Oriented {L}agrangian matroids},
  note       = {Combinatorial geometries (Luminy, 1999)},
  journal    = {European J. Combin.},
  fjournal   = {European Journal of Combinatorics},
  volume     = {22},
  year       = {2001},
  number     = {5},
  pages      = {639--656},
  issn       = {0195-6698,1095-9971},
  mrclass    = {05B35 (52B20 52B40 52C40)},
  mrnumber   = {1845489},
  mrreviewer = {Winfried\ Hochst\"attler},
  doi        = {10.1006/eujc.2000.0485},
  url        = {https://doi.org/10.1006/eujc.2000.0485},
}

@article{branden2007halfplane,
  author     = {Br{\"a}nd{\'e}n, Petter},
  title      = {Polynomials with the half-plane property and matroid theory},
  journal    = {Adv. Math.},
  fjournal   = {Advances in Mathematics},
  volume     = {216},
  year       = {2007},
  number     = {1},
  pages      = {302--320},
  issn       = {0001-8708,1090-2082},
  mrclass    = {05B35 (32E20 32F17 52A99)},
  mrnumber   = {2353258},
  mrreviewer = {David\ G.\ Wagner},
  doi        = {10.1016/j.aim.2007.05.011},
  url        = {https://doi.org/10.1016/j.aim.2007.05.011},
}

@article{branden2010discreteconcavity,
  author   = {Br{\"a}nd{\'e}n, Petter},
  title    = {Discrete concavity and the half-plane property},
  journal  = {SIAM J. Discrete Math.},
  fjournal = {SIAM Journal on Discrete Mathematics},
  volume   = {24},
  year     = {2010},
  number   = {3},
  pages    = {921--933},
  issn     = {0895-4801,1095-7146},
  mrclass  = {90C27 (15A42 30C15)},
  mrnumber = {2680224},
  doi      = {10.1137/090758738},
  url      = {https://doi.org/10.1137/090758738},
}

@misc{calvert2025quadraticexchangeequationscoxeter,
  title         = {Quadratic exchange equations for Coxeter matroids},
  author        = {Kieran Calvert and Aram Dermenjian and Alex Fink and Ben Smith},
  year          = {2025},
  eprint        = {2511.13498},
  archiveprefix = {arXiv},
  primaryclass  = {math.CO},
  url           = {https://arxiv.org/abs/2511.13498},
}

@misc{cheung2025valuateddeltamatroidsprincipal,
  title         = {Valuated Delta Matroids and Principal Minors of Hermitian matrices},
  author        = {Nathan Cheung and Tracy Chin and Gaku Liu and Cynthia Vinzant},
  year          = {2025},
  eprint        = {2507.16275},
  archiveprefix = {arXiv},
  primaryclass  = {math.CO},
  url           = {https://arxiv.org/abs/2507.16275},
}

@article{marcus2015interlacingi,
  author     = {Marcus, Adam W. and Spielman, Daniel A. and Srivastava,
                Nikhil},
  title      = {Interlacing families {I}: {B}ipartite {R}amanujan graphs of
                all degrees},
  journal    = {Ann. of Math. (2)},
  fjournal   = {Annals of Mathematics. Second Series},
  volume     = {182},
  year       = {2015},
  number     = {1},
  pages      = {307--325},
  issn       = {0003-486X,1939-8980},
  mrclass    = {05C50 (05C75)},
  mrnumber   = {3374962},
  mrreviewer = {A.\ Vijayakumar},
  doi        = {10.4007/annals.2015.182.1.7},
  url        = {https://doi.org/10.4007/annals.2015.182.1.7},
}

@article{marcus2015interlacingii,
  author     = {Marcus, Adam W. and Spielman, Daniel A. and Srivastava,
                Nikhil},
  title      = {Interlacing families {II}: {M}ixed characteristic polynomials
                and the {K}adison-{S}inger problem},
  journal    = {Ann. of Math. (2)},
  fjournal   = {Annals of Mathematics. Second Series},
  volume     = {182},
  year       = {2015},
  number     = {1},
  pages      = {327--350},
  issn       = {0003-486X,1939-8980},
  mrclass    = {46L05 (42A05 46B03 46L30)},
  mrnumber   = {3374963},
  mrreviewer = {Robert\ S.\ Doran},
  doi        = {10.4007/annals.2015.182.1.8},
  url        = {https://doi.org/10.4007/annals.2015.182.1.8},
}

@article{rincon2012isotropicallinearspaces,
  author     = {Rinc\'on, Felipe},
  title      = {Isotropical linear spaces and valuated {D}elta-matroids},
  journal    = {J. Combin. Theory Ser. A},
  fjournal   = {Journal of Combinatorial Theory. Series A},
  volume     = {119},
  year       = {2012},
  number     = {1},
  pages      = {14--32},
  issn       = {0097-3165,1096-0899},
  mrclass    = {52B40 (05B35 58A17)},
  mrnumber   = {2844079},
  mrreviewer = {Hannah\ Markwig},
  doi        = {10.1016/j.jcta.2011.08.001},
  url        = {https://doi.org/10.1016/j.jcta.2011.08.001},
}

@article{speyer2005horn,
  author     = {Speyer, David E.},
  title      = {Horn's problem, {V}innikov curves, and the hive cone},
  journal    = {Duke Math. J.},
  fjournal   = {Duke Mathematical Journal},
  volume     = {127},
  year       = {2005},
  number     = {3},
  pages      = {395--427},
  issn       = {0012-7094,1547-7398},
  mrclass    = {14P99 (05E15 15A42)},
  mrnumber   = {2132865},
  mrreviewer = {Thomas\ C.\ Craven},
  doi        = {10.1215/S0012-7094-04-12731-0},
  url        = {https://doi.org/10.1215/S0012-7094-04-12731-0},
}
\end{document}